\documentclass[11pt,twoside,letter]{article}
\usepackage[dvips]{graphicx}
\usepackage{fancyhdr}
\usepackage{amsmath,epsfig,amssymb}
\usepackage[pdftex]{color}
\usepackage{hyperref}

\usepackage{rawfonts}
\input prepictex
\input pictex
\input postpictex
\newcommand{\R}{\mathbb{R}}
\newcommand{\Z}{\mathbb{Z}}
\newcommand{\calS}{{\cal S}}
\newcommand{\calF}{{\cal F}}
\newcommand{\calM}{{\cal M}}
\newcommand{\Q}{\mathbb{Q}}
\newtheorem{theorem}{Theorem}[section]
\newtheorem{question}[theorem]{Question}
\newtheorem{lemma}[theorem]{Lemma}
\newtheorem{remark}[theorem]{Remark}
\newtheorem{proposition}[theorem]{Proposition}
\newtheorem{definition}[theorem]{Definition}
\newtheorem{example}[theorem]{Example}
\def\@begintheorem#1#2{
 \par\noindent\bgroup{\sc #1\ #2. }\it\ignorespaces}
\def\@beginremark#1#2{
 \par\noindent\bgroup{\sc #1\ #2. }\ignorespaces}
\def\@opargbegintheorem#1#2#3{
 \par\bgroup{\sc #1\ #2\ (#3). }\it\ignorespaces}
\def\@endtheorem{\egroup}

\renewcommand{\span}{ {\rm span}\, }
\newcommand{\spa}{{\rm \overline{span} \,}}
\newcommand{\vol}{ {\rm vol }\, }
\newcommand{\supp}{{\rm supp}\, }
\DeclareMathOperator*{\esssup}{ess\,sup}

\begin{document}

%
%
\title{\bf\vspace{-39pt} A time-frequency density criterion for operator identification}

%
%

\author{Niklas Grip \\ \small Department of Mathematics, Lule{\aa} University of Technology \\
\small 97187 Lule{\aa}, Sweden \\ \small niklas.grip@ltu.se\\
\\
G\"{o}tz E. Pfander \\ \small School of Engineering and Science, Jacobs University Bremen \\
\small 28759 Bremen, Germany \\ \small g.pfander@jacobs-university.de\\
\\
Peter Rashkov \\ \small Fachbereich Mathematik und Informatik, Philipps-Universit\"at Marburg \\
\small 35032 Marburg, Germany \\ \small rashkov@mathematik.uni-marburg.de}

%
%
\date{}

%
%
\maketitle
\thispagestyle{fancy}

%

%
%

\begin{abstract}
We establish  a necessary   density  criterion for the identifiability of time-frequency structured   classes of Hilbert-Schmidt operators. The density condition  is based on the density criterion for Gabor frames and Riesz bases in the space of square integrable functions.  We complement our findings with examples of identifiable operator classes. \vspace{5mm} 
\vspace{3mm}
\end{abstract}

%
%

\section{Introduction}

The goal in operator identification is to recover an incompletely known operator
from its action on a single input signal \cite{KP06,PW06}. In mathematical terms: for a normed linear space of operators $\mathcal{Z}$ mapping a set $X$ into a normed linear space $Y$, we seek $g\in X$ such
that the induced evaluation map
\begin{equation*}
 \Phi_g:\mathcal{Z} \to Y,\qquad H\mapsto Hg,
\end{equation*}
is bounded and boundedly invertible on its range. In short, we require that for some $g\in X$ there exist positive constants $A$ and $B$ with
\begin{equation*}
A\|H\|_{\mathcal{Z}}\leq \|Hg\|_Y\leq B \|H\|_{\mathcal{Z}},\quad H\in \mathcal{Z}.
\end{equation*}
This and similar inequalities will be represented  by 
\begin{equation*}
\|Hg\|_Y\asymp\|H\|_{\mathcal{Z}},\quad H\in \mathcal{Z},
\end{equation*}
from now on.

Identification of operators is important in numerous applications. For example, in mobile radio communications it is desirable to identify an unknown channel operator prior to information transmission. In radar applications, information on a target is gained through analyzing its response to a known sounding signal.

In this paper we focus on Hilbert-Schmidt operators on the space of square integrable functions on $\R$, $L^2(\R)$. Hilbert-Schmidt operators on $L^2(\R)$ are formally given by
\begin{equation*}
 Hf(x) = \iint \eta_H(t,\nu)\,e^{2\pi i\nu (x-t)}\,f(x-t)\,d\nu\,dt =\iint \eta_H(t,\nu)\,\pi_1(t,\nu)f(x)\,d\nu\,dt\notag
\end{equation*}
with $\eta_H\in L^2(\R^2)$. The unitary time-frequency shift $\pi_d(\lambda)$ by $\lambda=(t,\nu)\in \R^{2d}$  is defined by
\begin{equation}\label{eqn:TFshift}
\pi_d(\lambda)f(x)=T_tM_\nu f(x)=e^{2\pi i\nu (x-t)}\,f(x-t), \quad f\in L^2(\R^d).
\end{equation}
The space of Hilbert-Schmidt operators $HS$ inherits the Hilbert space structure from $L^2(\R^2)$ by setting $\langle H, K\rangle_{HS}=\langle \eta_H, \eta_K\rangle_{L^2}$ and $\|H\|_{HS}=\|\eta_H\|_{L^2}$ \cite{FN03,GP08}.

For  $\lambda=(s ,\omega ; z , y )\in\R^4$ and $H_0$ Hilbert-Schmidt with spreading function $\eta_0$ we define the  operator $H_\lambda$ by \begin{eqnarray}\label{eqn:defineHlambda}
 \eta_{H_\lambda}=\pi_2(\lambda)\eta_0=T_{(s ,\omega )}M_{( z , y )}\eta_0.
\end{eqnarray} 
The central goal of this paper is to establish a density criterion on a not necessarily full-rank lattice $\Lambda$ for the identifiability of the closed linear span of 
\begin{equation}\label{eqn:definecHlambda}
  (H_0,\Lambda) =  \{H_\lambda\}_{\lambda\in\Lambda},
\end{equation}
that is, a necessary density condition on $\Lambda$ for the existence of $g$ with
\begin{equation}\label{eqn:identification}
 \|Hg\|_{L^2(\R)}\asymp \|H\|_{HS}, \quad H\in  \overline{\span}(H_0,\Lambda).
\end{equation}

The paper is structured as follows. Section~\ref{sect:preliminaries} recalls general facts on modulation spaces, on Gabor Riesz bases and frames, and on Hilbert-Schmidt operators. In Section~\ref{section:mainresult} we discuss the identification problem outlined above in detail and state our main result, Theorem~\ref{thm:mainresult}.  Theorem~\ref{thm:mainresult} is proven in Section~\ref{sect:proof}. Section~\ref{section:sufficient} provides some examples of identifiable  classes of Hilbert-Schmidt operators. Also, the design of identifiers for the here considered operator families is discussed in Section~\ref{section:sufficient}.

\section{Background}\label{sect:preliminaries}

This section reviews some basic properties of Gabor Riesz sequences and frames in the Hilbert space of square integrable functions $L^2(\R^d)$ and in the space of Hilbert-Schmidt operators on $L^2(\R^d)$.

A countable family of vectors $\{g_\lambda\}_{\lambda\in\Lambda}$ in a  Hilbert space $\mathcal{H}$ is called a \emph{Riesz sequence} if 
\[
 \|\{c_\lambda\}\|_{\ell^2(\Lambda)}\asymp \|\sum_{\lambda\in\Lambda} c_\lambda g_\lambda \|_{\mathcal{H}},\quad \{c_\lambda\}\in\ell^2(\Lambda),
\]
where $\ell^2(\Lambda)$ denotes the space of square summable sequences indexed by $\Lambda$.\footnote{Recall $\ell^2(\Lambda)=L^2(\Lambda,\nu)$ where $\nu$ is the counting measure on $\Lambda$.}

If only $\|\{c_\lambda\}\|_{\ell^2(\Lambda)}\geq A \|\sum_{\lambda\in\Lambda} c_\lambda g_\lambda \|_{\mathcal{H}}$, $ \{c_\lambda\}\in\ell^2(\Lambda)$, for some positive $A$, then we refer to  $\{g_\lambda\}_{\lambda\in\Lambda}$ as \emph{Bessel sequence}. A \emph{Riesz basis} is a Riesz sequence that spans $\mathcal{H}$. 

A countable family $\{g_\lambda\}_{\lambda\in\Lambda}$ is a \emph{frame} for $\mathcal{H}$ if
\begin{equation}\label{eq:frameineqGeneral}
\|f\|_{L^2(\R^d)}\asymp\big\|\{\langle
f,g_\lambda \rangle\}_{\lambda\in\Lambda}\big\|_{\ell^2(\Lambda)},\quad f\in \mathcal{H}.
\end{equation}

The set $\Lambda=M\Z^{2d}\subset\R^{2d}$  with $M$ being a (not necessarily full rank) real $2d\times 2d$ matrix is called {\it lattice}. A Gabor system $(g,\Lambda)$ in $L^2(\R^d)$ is the set of all time-frequency shifts \eqref{eqn:TFshift} of
the window function $g$ by elements $\lambda=(x,\omega)\in \Lambda$, that is,
\[(g,\Lambda)=\{g_\lambda=\pi_d(\lambda)g: \lambda\in\Lambda\}.\]
The set $(g,\Lambda)$ is called \emph{Gabor Riesz sequence}  if it is a Riesz sequence in $L^2(\R^d)$ and a  \emph{Gabor frame} if it is  a frame for $L^2(\R^d)$.

Below, we shall utilize modulation space theory as developed by Feichtinger and Gr\"ochenig \cite{Fei83, Fei89, FG96, Gro01}. Let  $\mathcal S(\R)$ denote the space of Schwarz functions on $\R$ and $\mathcal S'(\R)$ its dual of so-called tempered distributions. Let $v_s$, $s\in \R$, be the polynomial weight function $v_s(z)=(1+|z|)^s$. The weighted mixed-norm sequence space $\ell^1_s(\Lambda)$ contains all sequences $\{c_\lambda\}_{\lambda\in\Lambda}$ with the property that 
\begin{equation}\label{def:wss}
\|\{c_\lambda\}\|_{\ell^{1}_s(\Lambda)}=\sum_{\lambda\in\Lambda}|c_{\lambda}\,v_s(\lambda)|<\infty.
\end{equation}
 The weighted modulation space $M_s^1(\R)$ consists of
all those tempered distributions $f\in{\cal S}'(\R)$ with
\begin{equation*}
\|f\|_{M_s^1}=\int|\langle
f,\pi_1(z)\gamma\rangle\, v_s(z)|\,dz<\infty,
\end{equation*}
with $\gamma(x)=e^{-|x|^2}$, $x\in\R$, and
where we refer to 
\begin{equation}
V_\gamma f(z)= \langle
f,\pi_1(z)\gamma\rangle, \quad z\in\R^2,
\end{equation}
as {\it short-time Fourier transform} of $f$ with respect to the window $\gamma$.

The weighted modulation space $M_s^{\infty}(\R)$ consists of all  $f\in{\cal S}'(\R)$ with norm
\begin{equation}\label{eq:defnmodspaces1}
\|f\|_{M_s^{\infty}}=\sup_{z\in\R^{2}}|\langle
f,\pi_1(z)\gamma\rangle v_s(z)|<\infty.
\end{equation}
 Note that the dual space of $M_s^1$ is $M_{-s}^\infty$. If $s=0$,
we write simply $M^1$ and $M^\infty$. For a detailed treatment of the theory of
modulation spaces we refer to Chapters 11 and 12 of~\cite{Gro01}.
%

%

{\it Hilbert-Schmidt operators} on $L^2(\R)$ are in one-to-one correspondence to their kernel \cite[page 267]{Con90}, and, similarly, they can be represented by their time-varying impulse response $h_H$, their Kohn-Nirenberg symbol $\sigma_H$, and their spreading function $\eta_H$. In fact, formally,
\begin{equation}\label{eqn:operatorrepresentations}
 \begin{aligned}
 Hf(x) &=\int \kappa_H(x,y) f(y)\, dy = \int h_H(t,x)\,f(x-t)\,dt \\ &= \iint \eta_H(t,\nu)\,e^{2\pi i\nu (x-t)}\,f(x-t)\,d\nu\,dt = \int \sigma_H(x, \xi)e^{2\pi i x  \xi} \widehat f(\xi)\,d \xi.
\end{aligned}
\end{equation}
The functions $\kappa_H,h_H,\sigma_H,\eta_H$ are related   by
\begin{equation*}
\begin{aligned}
 \int \eta_H(t,\nu)\, e^{2\pi i \nu (x-t)}\, d\nu =h_H(t,x)=\kappa_H(x,x-t)
=\int \sigma_H (x, \xi)e^{2\pi i  \xi t}\, d \xi,
\end{aligned}
\end{equation*}
and 
\begin{equation*}
 \begin{aligned}
 \|H\|_{HS}=\|\kappa_H\|_{L^2(\R^2)}=\|h_H\|_{L^2(\R^2)}=\|\eta_H\|_{L^2(\R^2)}=\|\sigma_H\|_{L^2(\R^2)}.
 \end{aligned}
\end{equation*}
due to the unitarity of the $L^2$-Fourier transform   $\mathcal F$ which is densely defined by
\begin{equation*}
\mathcal F f ( \xi)=\widehat f( \xi)=\int f(x)e^{-2\pi i x \xi}\,dx.
\end{equation*}
%
%

For reference, we include the definition of  Beurling density of $\R^d$. Let $\mathcal B_d(R)$ denote the ball in $\R^{d}$ centered at 0 with radius
$R$ and    let $|\mathcal M|$ denote the cardinality of the set $\mathcal M$. The lower and
upper Beurling densities of $\Lambda\subseteq\R^d$ are given by
\begin{equation*}
\begin{aligned}
D^-(\Lambda)&=\liminf_{R\to\infty}\inf_{z\in \R^d}\frac{|\Lambda\cap
\{\mathcal B_d(R)+z\}|}{\vol \mathcal B_d(R)}, \\
D^+(\Lambda)&=\limsup_{R\to\infty}\sup_{z\in \R^d}\frac{|\Lambda\cap
\{\mathcal B_d(R)+z\}|}{\vol \mathcal B_d(R)}.
\end{aligned}
\end{equation*}
In case of $D^-(\Lambda)=D^+(\Lambda)$, we
say that $\Lambda$ has Beurling density $D(\Lambda)=D^-(\Lambda)=D^+(\Lambda)$.  Note that the Beurling density of a lattice $\Lambda$  equals
the inverse of the Lebesgue measure of any measurable fundamental domain of $\Lambda$.
See \cite{Kut06} for a more general concept of Beurling density.

 \section{Basic Observations and Main Result}\label{section:mainresult}

If  $ (H_0,\Lambda)$ in \eqref{eqn:definecHlambda} is a Riesz sequence in the space of Hilbert-Schmidt operators, then 
\begin{equation*}
 \overline{\span}(H_0,\Lambda) =\big\{\sum_{\lambda \in \Lambda}c_\lambda\,  H_\lambda\,: \ \ \{c_\lambda\}\in\ell^2(\Lambda)  \big\}
\end{equation*}
and
\begin{equation}\label{eqn:HS-Riesz}
\big\|\sum_{\lambda\in\Lambda} c_\lambda
H_\lambda\big\|_{HS}\asymp\|\{c_\lambda\}\|_{\ell^2(\Lambda)},\quad \{c_\lambda\}\in\ell^2(\Lambda).
\end{equation}
In this case,  identifiability of $\overline{\span}(H_0,\Lambda)$ by $g$ is equivalent to establishing
\begin{equation} \label{eqn:equivToIdent}
\big\|\sum_{\lambda\in\Lambda} c_\lambda H_\lambda
g\big\|_{2}\asymp\|\{c_\lambda\}\|_{\ell^2(\Lambda)},\quad \{c_\lambda\}\in\ell^2(\Lambda),
\end{equation}
that is, to showing that $\{H_\lambda g\}_{\lambda\in\Lambda}$ is a Riesz sequence in $L^2(\R)$. Note that \eqref{eqn:HS-Riesz} corresponds to a Riesz sequence condition in the ``large space'' $L^2(\R^2)$, while \eqref{eqn:equivToIdent} is a Riesz sequence condition on a similarly structured family of vectors in the ``smaller space'' $L^2(\R)$.  We shall generally assume  that the weaker condition,  $(H_0,   \Lambda)$ is  a Riesz  sequence, holds and  focus on the question whether the set $\{H_\lambda g\}_{\lambda\in \Lambda}$ is a Riesz sequence for some $g$. 
%
\begin{remark}
 \rm ({\it i})\  We established  that  \eqref{eqn:HS-Riesz} and \eqref{eqn:equivToIdent} implies \eqref{eqn:identification}.  Moreover, if $g$ is square integrable, then $$\|\sum_{\lambda\in\Lambda} c_\lambda H_\lambda g\|_{L^2} \leq \|\sum_{\lambda\in\Lambda} c_\lambda H_\lambda\|_{HS} \,\|g\|_{L^2},\quad \{c_\lambda\}\in\ell^2(\Lambda),$$ and \eqref{eqn:HS-Riesz} can be replaced with the condition that $(H_0,\Lambda)$ is a Bessel sequence 
 to obtain \eqref{eqn:identification}. 
 This argument is not always applicable as we shall generally allow $g$ to be a tempered distribution  and use the fact that  some spaces of operators map spaces of distributions into $L^2(\R)$. 
For example, operators in so-called operator Paley-Wiener space
\begin{equation*}
 OPW^2(S)=\{H\in HS(L^2(\R)):\ \supp \eta_H\subseteq S \}
\end{equation*}
map boundedly the modulation space $M^\infty(\R)$ defined in \eqref{eq:defnmodspaces1} to $L^2(\R)$ in case that $S$  is compact  \cite{Pfa07c}.

\noindent ({\it ii})\
The identifiability of $\overline{\span}(H_0, \Lambda)$  neither implies \eqref{eqn:HS-Riesz} nor \eqref{eqn:equivToIdent}. Indeed, in some cases  $\overline{\span}(H_0,\Lambda)$ is identifiable, $\overline{\span}(H_0,\Lambda)=\overline{\span}(H_0, 2 \Lambda)$, and $(H_0,  2 \Lambda)$ and  $\{H_\lambda g\}_{\lambda\in2 \Lambda}$  are Riesz sequences while $(H_0,   \Lambda)$ and  $\{H_\lambda g\}_{\lambda\in \Lambda}$ are not.

\end{remark}
%

The framework developed here is motivated in part by the following  well known results from time-frequency analysis.  The first result gives a necessary condition on a Gabor system to form a Riesz sequence (see \cite{Hei07} and references therein).

\begin{theorem}\label{thm:identification1stExample} Let $\Gamma = 
\big(\begin{smallmatrix}
 a_1 & a_2 \\ b_1 & b_2
\end{smallmatrix}\big)\Z^2
$ be a full rank lattice in $\R^2$. If there exists $g\in L^2(\R)$ such that
 $(g,\Gamma)=\{\pi_1(\gamma) g\}_{\gamma\in\Gamma}$ is a Riesz sequence in $L^2(\R)$, then the Beurling density $D(\Gamma) = \big|\det \big(\begin{smallmatrix}
 a_1 & a_2 \\ b_1 & b_2
\end{smallmatrix}\big)\big|^{-1}$ of $\Gamma$ satisfies $D(\Gamma)\leq 1$.
\end{theorem}

Theorem~\ref{thm:identification1stExample} is a special case of the herein established framework of  operator identification.  Indeed, set 
$$\Lambda_\Gamma=
 \begin{pmatrix}
	1&0& 0& 0\\
	0&1&1&0
 \end{pmatrix}^T	\Gamma = 
  \begin{pmatrix}
	a_1&b_1& b_1 & 0\\
	a_2&b_2&b_2 &0
 \end{pmatrix}^T	\Z^2
$$
and observe that Lemma~\ref{lem:con} implies that for any Hilbert-Schmidt operator $H_0$ we have
\begin{equation*}
  \overline{\span}(H_0,\Lambda_\Gamma) = \overline{ \span} \{\pi_1(\gamma)\circ H_0\}_{\gamma\in\Gamma} .
\end{equation*}
So $g$ identifies  $ \overline{\span}(H_0,\Lambda_\Gamma) $   if and only if $\{ \pi_1(\gamma) (H_0 g) \}_{\gamma \in \Gamma}$ is a Riesz sequence. It follows that $D(\Gamma)\leq 1$ is necessary for the identifiability of $ \overline{\span}(H_0,\Lambda_\Gamma) $.\footnote{Our reasoning uses $Hg\in L^2(\R)$ and not $g\in L^2(\R)$. Indeed, the flexibility of choosing not square integrable $g$ is quite beneficial. In fact, for some $(H_0,\Lambda)$ we must choose $g\in M^\infty(\R)\setminus L^2(\R)$ to achieve that $\{H_\lambda g\}_{\lambda \in \Lambda}$ is a Riesz sequence in $L^2(\R)$.} The condition $D(\Gamma)\leq 1$ implies that $\Lambda_\Gamma$ is not too dense in the two-dimensional ``tilted'' plane $\R\times\{y=z\}\times\{0\}\subseteq\R^4$. 

%

The second result motivating this paper plays a critical role in the analysis of slowly time-varying operators \cite{Kai59,Bel69} and in the recently developed sampling theory for operators \cite{KP06,PW06,Pfa11}.

\begin{theorem}\label{thm:identification2ndExample}
Let $\mathcal M = \big(\begin{smallmatrix}
 a_1 & a_2 \\ b_1 & b_2
\end{smallmatrix}\big)\Z^2 \subseteq \R^2$ be a full-rank lattice and $H_0:L^2(\R)\longrightarrow L^2(\R)$, $H_0f=\rho\cdot (f\ast r)$, be a product-convolution Hilbert-Schmidt operator with $\rho,\widehat r$ smooth and compactly supported. If there exists a tempered distribution $g$ such that
 $\{\pi_1(\gamma)H_0\pi_1(\gamma)^\ast g\}_{\gamma\in\mathcal M}$ is a Riesz sequence in $L^2(\R)$, then $D(\mathcal M)\leq 1$.
\end{theorem}
%

%

Similarly to above, Lemma~\ref{lem:con} implies that setting 
\begin{eqnarray*}
\Lambda_{\mathcal M}=\begin{pmatrix} 0&0& 1& 0\\
0&0&0&1 
\end{pmatrix}^T\mathcal M = \begin{pmatrix} 0&0& a_1& b_1\\
0&0&a_2&b_1 
\end{pmatrix}^T \Z^{2}
\end{eqnarray*}
indicates that Theorem~\ref{thm:identification2ndExample} may be considered a special case of  our framework.

 Theorems \ref{thm:identification1stExample} and  \ref{thm:identification2ndExample} motivate the central question in this paper which we paraphrase as follows.
\begin{question} \rm
Can we define a density $\widetilde{D}$ on lattices $\Lambda=M\Z^2\subseteq \R^4$ 
so that for a positive constant $C$ we have   $\widetilde{D}(\Lambda)> C$ implies $ \overline{\span}(H_0,\Lambda)$ is not identifiable whenever $(H_0,\Lambda)$ is a Riesz sequence?
%
\end{question}
%

We choose the following ``Beurling-type'' density for sets of points $\Lambda$ lying within general two-dimensional subspaces of $\R^{4}$.

\begin{definition}\label{def:beruling2d}
 The
``two-dimensional'' upper and lower Beurling densities $($or for short
2-Beurling density$)$ of $\Lambda\subseteq \R^4$ are given by
\begin{equation}\label{eqdef:beruling2d}
\begin{aligned}
D_{(2)}^-(\Lambda)&=\sup_{U\in\mathcal U}\ \liminf_{R\to\infty}\ \inf_{z\in
U}\frac{|\Lambda\cap \{\mathcal B_{4}(R)+z\}|}{\pi R^{2}}, \\
D_{(2)}^+(\Lambda)&=\sup_{U\in\mathcal U}\ \limsup_{R\to\infty}\ \sup_{z\in
U}\frac{|\Lambda\cap \{\mathcal B_{4}(R)+z\}|}{\pi R^{2}},
\end{aligned}
\end{equation}
where $\mathcal U$ denotes the set of two-dimensional affine subspaces of $\R^4$ and $\mathcal B_4(R)$ is the centered open ball with radius $R$ in $\R^4$.

If $D_{(2)}^+(\Lambda)= D_{(2)}^-(\Lambda)$, then $\Lambda$ has uniform 2-Beurling density $D_{(2)}(\Lambda)=D_{(2)}^-(\Lambda)$.
\end{definition}

Observe that with
\begin{equation*}
 \Lambda=\begin{pmatrix} a_1&b_1& c_1& d_1\\
a_2&b_2& c_2& d_2
\end{pmatrix}^T\Z^2
\end{equation*}
we have
\begin{equation}\label{eq:2density}
\begin{aligned}D_{(2)}(\Lambda)=&\big[(a_1b_2-a_2b_1)^2+(a_1c_2-a_2c_1)^2\\ &\quad +(a_1d_2-a_2d_1)^2+(b_1c_2-b_2c_1)^2\\
&\quad +(b_1d_2-b_2d_1)^2+(c_1d_2-c_2d_1)^2\big]^{-1/2}.
\end{aligned}
\end{equation}
Hence, for
 $
\Lambda_{\mathcal M}=\big(\begin{smallmatrix} 0&0& a_1& b_1\\
0&0&a_2&b_2
\end{smallmatrix}\big)^T\Z^{2}
$ in Theorem~\ref{thm:identification2ndExample}
we have $D_{(2)}(\Lambda_{\mathcal M})=|a_1b_2-a_2 b_1|^{-1}=\big|\det \big( 
\begin{smallmatrix}
 a_1 & a_2\\
 b_1 & b_2
\end{smallmatrix}\big)\big|^{-1}$, and for $\displaystyle
\Lambda_{\Gamma}=\big(\begin{smallmatrix}  a_1&b_1& b_1& 0\\
a_2&b_2&b_2&0
\end{smallmatrix}\big)^T\Z^{2}
$
in Theorem~\ref{thm:identification1stExample}
we have $D_{(2)}(\Lambda_{\Gamma})={ 2^{-1/2}}\,\big|\det \big( 
\begin{smallmatrix}
 a_1 & a_2\\
 b_1 & b_2
\end{smallmatrix}\big)\big|^{-1}$.

The main result in this paper provides a necessary condition on $\Lambda$ for the existence of $g$ so that $\{ H_\lambda g\}_{\lambda \in\Lambda}$ is a Riesz sequence.  $M_s^1$ and $M^\infty$ denote modulation spaces whose definitions we recalled in Section~\ref{sect:preliminaries}.
\begin{theorem}\label{thm:mainresult}
Let $\Lambda=M\Z^2\subseteq \R^4$ and $H_0$ be an operator with $\eta_{H_0}\in M_s^1(\R)$, $s>2$. If $(H_0,\Lambda)$ is Riesz and 
$ \overline{\span}(H_0,\Lambda)$ is identifiable  by $g\in M^\infty(\R)$  then $D_{(2)}(\Lambda)\leq \sqrt{2}$.
%
\end{theorem}
{ If the first or the last row of $M$ is zero, then the bound on $D_{(2)}(\Lambda)$ can be improved. Identifiability then implies $D_{(2)}(\Lambda)\leq 1$. This is the case in Theorem \ref{thm:identification2ndExample} and Theorem \ref{thm:identification1stExample} respectively (in Theorem~\ref{thm:identification1stExample} even $D_{(2)}(\Lambda)\leq 2^{-1/2}$). 
It is not clear whether the constant $\sqrt 2$ in Theorem~\ref{thm:mainresult} can be improved in the general case.

Identifiability of $\overline{\span} (H_0,\Lambda)$ depends on both, $H_0$ and the lattice $\Lambda$, so it is not surprising that there exists no $c>0$ with
$
 D_{(2)}^+(\Lambda)< c$ implies $ \overline{\span}(H_0,\Lambda)$ is identifiable.
In fact, it is easy to construct for  any $\epsilon >0$ a Riesz sequence $(H_0,\Lambda)$ with $D_{(2)}(\Lambda)\leq \epsilon$ but  $ \overline{\span}(H_0,\Lambda)$ is not identifiable (see Example~\ref{prop:casenotidef} and \cite{GPR10}). 

We would like to emphasize    that the  operator outputs considered herein are in the ``small'' space $L^2(\R)$ while the kernels and spreading functions of the operators are  in the ``larger'' space of $L^2(\R^2)$ functions. This dimension mismatch 
implies that a single evaluation map $H\mapsto Hg$ cannot identify the space of  Hilbert-Schmidt operators as a whole, just as a degree of freedom counting argument shows that the space of complex $n\times n$ matrices requires the use  of $n$ identifiers  for identification. 


%

\section{Proof of Theorem~\ref{thm:mainresult} }\label{sect:proof}

The following lemmas are used in the proof of Theorem \ref{thm:mainresult}.
\begin{lemma}\label{lem:con}
{ Let $\eta_0$ denote} the spreading function of $H_0\in HS(\R)$. Then the operator $T_{a}M_{b}T_{-c}H_0T_{c}M_{d}$
   has spreading function
 \begin{equation*}
 \eta_{T_{a}M_{b}T_{-c}H_0T_{c}M_{d}}=T_{a,b+d}M_{b,c}\eta_0,
 \quad a,b,c,d\in\R.
\end{equation*}
 Hence, 
\begin{equation*}
\eta_{H}=
\sum_{\lambda=(s ,\omega , \xi , y )\in\Lambda}c_{\lambda}\eta_{H_\lambda}=\sum_{\lambda=(s ,\omega , \xi , y )\in\Lambda}c_{\lambda}\pi_2(\lambda)\eta_0.
\end{equation*}
converges in $L^{2}$-norm if and only if
 \begin{eqnarray*}\notag
 H&=&\sum_{\lambda=(s ,\omega , \xi, y )\in\Lambda}c_{\lambda}
   T_{s }M_{ \xi }T_{- y }
   H_0
   T_{ y }M_{\omega - \xi }
 \end{eqnarray*}
converges in $HS(\R)$.
\end{lemma}
\begin{proof}
A change of variables $t=t-a$, $\nu=\nu-b-d$ and 
  the relation
  $T_{x}M_{\omega}=e^{-2\pi i x\omega}M_{\omega}T_{x}$ implies 
  \begin{equation*}
    \begin{aligned}
      \iint &
            T_{a,b+d}M_{b,c}\eta_0
            (t,\nu)f({x-t})e^{2\pi i\nu(x-t)} \,d \nu\,dt=\\
        &=
            \iint 
            e^{2\pi i(b (t-a) + c(\nu-b-d))}
            \eta_0(t-a,\nu-b-d)
            f({x-t})e^{2\pi i\nu(x-t)} \,d \nu\,dt\\
        &=
            \iint 
            e^{2\pi i(b t + c\nu)}
            \eta_0(t,\nu)
            f({x-t-a})
            e^{2\pi i(\nu+b+d)(x-t-a)} \,d \nu\,dt\\
        &=
            \iint 
            e^{2\pi i(b t + c\nu)}
            \eta_0(t,\nu)
            T_{t+a}M_{\nu+b+d}f(x)\,d \nu\,dt\\
        &=  T_{a}M_{b}T_{-c}
            \iint 
            \eta_0(t,\nu)
            T_{t}M_{\nu}T_{c}M_{d}f(x)\,d \nu\,dt\\
        &=
            T_{a}M_{b}T_{-c}
            H_0
            T_{c}M_{d}f(x).
    \end{aligned}
  \end{equation*}
  Hence, in particular,
  \begin{equation*}
    \begin{aligned}
      (H f)(x)
        =&\iint 
            \eta_{H}(t,\nu)f(x-t)e^{2\pi i\nu(x-t)} \,d \nu\,dt\\
        =&\sum_{\lambda=(s ,\omega ,\xi , y )\in\Lambda}c_{\lambda}
            \iint             T_{s,\omega}M_{\xi,y}
            \eta_0(t,\nu)f(x-t)e^{2\pi i\nu(x-t)} \,d \nu\,dt\\
        =&\sum_{\lambda=(s ,\omega ,\xi , y )\in\Lambda}c_{\lambda}
            T_{s}M_{\xi}T_{-y}
            H_0
            T_{y}M_{\omega-\xi}f(x).
    \end{aligned}
  \end{equation*}
\end{proof}
\begin{lemma}\label{lemma:lemma1}
Let $p,q\in C_c^\infty(\R^d)$ and let $P$ be the product-convolution operator with
spreading function $\eta_P=p\otimes q$. Then there exist functions
$\psi_1,\psi_2\in\calS(\R^d)$ such that   
$$
|Pf(x)|\le \|f\|_{M^\infty(\R^d)}|\psi_1(x)|,\ \  
|\calF Pf(\omega)|\le \|f\|_{M^\infty(\R^d)}|\psi_2(\omega)|,\ \ f\in M^\infty(\R^d).$$
\end{lemma}

Lemma~\ref{lemma:lemma1} is a straightforward generalization of Lemma 3.4 in \cite{KP06} and its proof is omitted.
Lemma~\ref{lemma:lemma1} {with $d=1$} is used to prove the following {result}.

\begin{lemma}\label{lm:estimate}
For $H$   with spreading function $\eta_0\in M_s^1(\R)$, $s>2$, {there} exist $\varphi_1(t) ,\varphi_2(t) =O(t^{-s})$ with 
$$
	 |Hf(x)|\le \|f\|_{M^\infty(\R)}\varphi_1(x),\ \
|\calF Hf(\omega)|\le \|f\|_{M^\infty(\R)}\varphi_2(\omega),\ \ f\in M^\infty(\R).
$$
\end{lemma}
 The decay estimates show that $Hg\in
L^2(\R)$ \cite[(2.52)]{Fol99}.
\begin{proof}
{By choosing parameters $a,b,c,d$ with $ac<1,bd<1$ and functions $p,q\in C_c^\infty(\R)$ with $[-\frac a2,\frac a2]\subset\supp p\subset[-\frac 1{2c},\frac1{2c}],[-\frac b2,\frac b2]\subset\supp q\subset[-\frac 1{2d},\frac1{2d}]$} we obtain a Gabor frame $(\eta_P, a\Z\times b\Z\times c\Z\times
d\Z)$ for $L^2(\R^{2})$, where $\eta_P=p\otimes q$ (see \cite[pages 22-23]{PR10} or \cite{PRW12}). 
Because $\eta_P\in\calS(\R^{2})\subset
M_s^1(\R^{2})$, the Gabor system $(\eta_P, a\Z\times
b\Z\times c\Z\times d\Z)$ is a universal Banach frame according to
the definition of Gr\"ochenig for all modulation spaces
$M_s^1(\R^{2})$ \cite[Section 13.6]{Gro01}. 
{The main result from \cite{GL03} states that the canonical dual window satisfies $\tilde\eta_P\in M_s^1(\R^{2})$.
Let us denote by $\ell^1_s(\Z^4)$ the weighted mixed-norm sequence space containing all sequences $\{c_z\}_{z\in\Z^4}$ such that the norm 
\begin{equation*}
\|\{c_z\}\|_{\ell^{1}_s(\Z^4)}=\sum_{z\in\Z^4}|c_{z}\,v_s(z)|<\infty,
\end{equation*}
where $v_s$ is defined by \eqref{def:wss}.}

Then the result from \cite[Corollary 12.2.6]{Gro01} shows that the series expansion
\begin{equation}\label{eq:series-expansion-M_s}
 f=\sum_{k,l,m,n\in\Z}\langle f, T_{ak,bl}M_{cm,dn}\tilde\eta_P\rangle
 T_{ak,bl}M_{cm,dn}\eta_P,\quad f \in M_s^1(\R^2)
\end{equation}
is convergent in the $M_s^1$-norm and  
\begin{equation}\label{eq:frame-inequality-M_s}
\|f\|_{M_s^1(\R)}\asymp \|\{\langle f,
T_{ak,bl}M_{cm,dn}\tilde\eta_P\rangle\}\|_{\ell_{s}^1(\Z^4)}, \quad f\in M_s^1(\R^2), \end{equation}
Furthermore, as the coefficients have
decay stronger than $\ell^1(\Z^4)$, the convergence of the series holds in
$L^1(\R^2)$ and in $L^2(\R^2)$.

Since the operator $H$ has spreading function $\eta_H \in
M_s^1(\R^{2})$, \eqref{eq:series-expansion-M_s}
and~\eqref{eq:frame-inequality-M_s} imply 
\[\eta_H=\sum_{k,l,m,n\in\Z}c_{k,l,m,n}T_{ak,bl}M_{cm,dn}\eta_P\]
for some $\{c_{k,l,m,n}\}\in\ell_s^1(\Z^4)$. It is legitimate to use as identifier distributions $g\in
M^\infty(\R)$, because of the inclusion
\[\calS(\R)\subset M_s^1(\R)\subset M^1(\R)\subset M^\infty(\R)\subset M^{\infty}_{-s}(\R)\subset\calS'(\R).\]
This inclusion is a consequence of~\cite[Proposition 11.3.4]{Gro01} and \cite[Corollary 12.1.10]{Gro01}. The fact that the
constant weight 1 is $v_s$-moderate (\cite[Lemma
11.1.1]{Gro01}) provides the inclusion $M_s^1(\R)\subset
M^1(\R)$.

Next, we estimate the decay of $Hg$ in the time and frequency
domain. 
We shall use the fact that translation and modulation are
isometries on $M^1(\R)$ and, hence, also on $M^\infty(\R)$ and estimate
\begin{equation}\label{eq:phi-1}
\begin{aligned}
 |Hg(x)| &= |\sum_{k,l,m,n\in\Z}c_{k,l,m,n}T_{ak}M_{-cm}T_{-dn}PT_{dn}M_{bl-cm}\,g(x)| \\
 &\le \sum_{k,l,m,n\in\Z}|c_{k,l,m,n}|\cdot T_{ak-dn}|PT_{dn}M_{bl-cm}\,g(x)| \\
 &\le\|g\|_{M^\infty(\R)}
 \sum_{k,l,m,n\in\Z}|c_{k,l,m,n}|\cdot T_{ak-dn}\phi_1(x),
\end{aligned}
\end{equation}
where $\phi_1\in\calS(\R)$ is some positive function, such that $|PT_{dn}M_{bl-cm}g(x)|\le\|g\|_{M^\infty(\R)}\phi_1(x)$ after Lemma~\ref{lemma:lemma1}.
For the sake of clarity we denote the expression on the right-hand side
of~\eqref{eq:phi-1} by
\[\Phi_1(x)=\sum_{k,l,m,n\in\Z}|c_{k,l,m,n}|\,T_{ak-dn}\phi_1(x)=\sum_{k,n\in\Z}\tilde{c}_{k,n}T_{ak-dn}\phi_1(x),\]
{where $\tilde c_{k,n}=\sum_{l,m\in\Z}|c_{k,l,m,n}|$.}

We claim that {$\Phi_1(x)=O(x^{-s})$ due to
$\{|c_{k,l,m,n}|\}_{k,l,m,n\in\Z}\in\ell_s^1(\Z^{4})\subset\ell^1(\Z^4)$.} 
Let us make a change of variables $|x|^{s}=y^2,y\ge0$, which is equivalent to 
$x=y^{\frac2s}$, $x\ge0$, and $x=-y^{\frac2s}$, $x<0$. Then
\begin{equation}\label{eq:star1}
\sup_{x\ge0} |x^{s}|\cdot|
\sum_{k,n\in\Z}\tilde{c}_{k,n}T_{ak-dn}\phi_1(x)|=\sup_{y\ge0} |y^2
\sum_{k,n\in\Z}\tilde{c}_{k,n}T_{ak-dn}\phi_1(y^{\frac 2{s}})|.
\end{equation}
Since $y^{\frac 2{s}}$ is monotone on $[0,\infty)$ and due to our
choice $\phi_1\in\calS(\R)$ (that is, $\phi_1$ decays faster than the
reciprocal of any polynomial on $\R$),
$\tilde\phi(y)=\phi_1(y^{\frac 2{s}})$ also decays faster
than the reciprocal of any polynomial. 
Then we can estimate $\sup_{y\ge0} |y^2 \sum_{k,n}\tilde{c}_{k,n}T_{ak-dn}\tilde\phi(y)|$ by setting $\Psi_1(y)=y\tilde\phi(y),\Psi_2(y)=y^2\tilde\phi(y)$ and using the equality $y^2=(y-ak-dn)^2-(ak-dn)^2+2y(ak-dn)$. 
Hence, by the triangle inequality it follows for \eqref{eq:star1} that
\begin{equation}\label{eqn:eqan4n}
\begin{aligned}
\Big|y^2 \sum_{k,n\in\Z}\tilde{c}_{k,n}T_{ak-dn}\tilde\phi(y)\Big|&\le \Big|\sum_{k,n\in\Z}\tilde{c}_{k,n}T_{ak-dn}\Psi_2\Big|\\
&\quad+ \Big|\sum_{k,n\in\Z}\tilde{c}_{k,n}(ak-dn)^2T_{ak-dn}\tilde\phi\Big| \\
 &\quad + 2\Big|\sum_{k,n\in\Z}\tilde{c}_{k,n}(ak-dn)T_{ak-dn}\Psi_1\Big|.
 \end{aligned}
\end{equation}
Taking the supremum in \eqref{eqn:eqan4n} leads to
\begin{equation}\label{eqn:eqan4}
\begin{aligned}
\sup_{y\ge0} \Big|y^2 \sum_{k,n\in\Z}\tilde{c}_{k,n}T_{ak-dn}\tilde\phi(y)\Big|&\le \sup_{y\ge0}\Big|\sum_{k,n}\tilde{c}_{k,n}T_{ak-dn}\Psi_2\Big|\\
&\quad+\, \sup_{y\ge0} \Big|\sum_{k,n\in\Z}\tilde{c}_{k,n}(ak-dn)^2T_{ak-dn}\tilde\phi\Big| \\
 &\quad +\, 2\sup_{y\ge0}\Big|\sum_{k,n\in\Z}\tilde{c}_{k,n}(ak-dn)T_{ak-dn}\Psi_1\Big| .
 \end{aligned}
\end{equation}
We compute the following upper estimate of the summands on the right-hand side in~\eqref{eqn:eqan4}
\begin{equation}\label{eqn:eqan5}
\begin{aligned}
\sup_{y\ge0} \Big|y^2 \sum_{k,n\in\Z}\tilde{c}_{k,n}T_{ak-dn}\tilde\phi(y)\Big| &\le \sum_{k,n\in\Z}|\tilde{c}_{k,n}|\sup_{y\ge0} |\Psi_2(y)| \\
&\quad+ \sum_{k,n\in\Z}|\tilde{c}_{k,n}(ak-dn)^2|\sup_{y\ge0}|\tilde\phi(y)|\\
 &\quad+ 2\sum_{k,n\in\Z}|\tilde{c}_{k,n}(ak-dn)|\sup_{y\ge0}|\Psi_1(y)|.
 \end{aligned}
\end{equation}

We analyze separately the three
summands from \eqref{eqn:eqan5}. Since $\tilde\phi,\Psi_1,\Psi_2$ belong to the Schwarz
class, they are bounded and decay faster than the reciprocal of any
polynomial. 
Also {the fact that
$\{c_{k,l,m,n}\}_{k,l,m,n\in\Z}\in\ell^1_s(\Z^4)$ implies the existance of constants $C_1,C_2>0$} such that 
\begin{equation*}
\begin{aligned}
 \sum_{k,n\in\Z}|\tilde{c}_{k,n}|\cdot|ak-dn|^2&\le C_1
\sum_{k,n\in\Z}|\tilde{c}_{k,n}|(1+a|k|+d|n|)^2\\
&<\|\{c_{k,l,m,n}\}_{k,l,m,n\in\Z}\|_{\ell_s^1(\Z^4)}<\infty, \\
 \sum_{k,n\in\Z}|\tilde{c}_{k,n}|\cdot|ak-dn|&\le C_2\sum_{k,n\in\Z}|\tilde{c}_{k,n}|(1+a|k|+d|n|)\\
&<\|\{c_{k,l,m,n}\}_{k,l,m,n\in\Z}\|_{\ell_s^1(\Z^4)}<\infty.
\end{aligned}
\end{equation*}
Furthermore,
\begin{equation*}
 \sum_{k,n\in\Z}|\tilde{c}_{k,n}|\le \|\{c_{k,l,m,n}\}_{k,l,m,n\in\Z}\|_{\ell_s^1(\Z^4)}<\infty.
\end{equation*}

Thus the  expression on the right-hand side of~\eqref{eqn:eqan5} is bounded, implying the desired decay rate of $Hg$ for $x>0$. 
In a similar fashion we prove the decay for $x<0$. 
Thus $\sup_{x\in\R}|x^{s}\Phi_1(x)|<C$ and $|Hg(x)|\le\|g\|_{M^\infty(\R)}\varphi_1(x)$ has decay $O(x^{-s}),s>2$.

A similar estimate can be done for the decay of the Fourier transform of $Hg$
\begin{equation}\label{eq:phi-2}
\begin{aligned}
|\calF{Hg}(\omega)| &= |\sum_{k,l,m,n\in\Z}c_{k,l,m,n}M_{-ak}T_{-cm}M_{dn}\calF PT_{dn}M_{bl-cm}g(\omega) |\\
 &\le \sum_{k,l,m,n\in\Z}|c_{k,l,m,n}|\cdot T_{-cm} |\calF PT_{dn}M_{bl-cm}g(\omega) | \\
 &\le
 \|g\|_{M^\infty(\R)}\sum_{k,l,m,n\in\Z}|c_{k,l,m,n}|\cdot T_{-cm}\phi_2(\omega),
\end{aligned}
\end{equation}
where $\phi_2\in\calS(\R)$ is some positive function, such that $|\calF PT_{dn}M_{bl-cm}g(\omega) |\le\|g\|_{M^\infty(\R)}\phi_2(\omega)$ after Lemma~\ref{lemma:lemma1}.
We denote the expression on the right-hand side of ~\eqref{eq:phi-2}, by
\[\Phi_2(x)=\sum_{k,l,m,n}|c_{k,l,m,n}|\,T_{-cm}\phi_2(x),\]
and prove in a similar fashion that $\Phi_2(x)=O(x^{-s}),s>2$.
\end{proof}

 \textit{Proof of Theorem} \ref{thm:mainresult}.
For $m,n\in\Z$ and $\lambda=M(m,n)^T $, we observe that
\begin{eqnarray*}
 H_\lambda &=& T_{(a_1-d_1) m+(a_2-d_2)n}M_{c_1 m+c_2n}
 \\ && \quad H_0 T_{d_1 m+d_2n} M_{(b_1-c_1) m+(b_2-c_2)n}.
\end{eqnarray*}
%
Set
$ g_{m,n} =H_0 T_{d_1 m+d_2n} M_{(b_1-c_1) m+(b_2-c_2)n}g$. 
Since  $\eta_{H_0}\in M^1_s(\R)$, $s>2$, Lemma~\ref{lm:estimate} implies the existence of $\phi_1(x)=O(x^{-s}),\phi_2(\omega)=O(\omega^{-s}),s>2$
such that
\begin{equation}\label{eq:estimates1}
\begin{aligned}
|g_{m,n}(x)|&
<\phi_1(x)\|g\|_{M^\infty(\R)},\\
|\calF g_{m,n}(\omega)|&
<\phi_2(\omega)\|g\|_{M^\infty(\R)}.
\end{aligned}
\end{equation}
To prove the claim of the theorem, we show that the family
\[\{H_\lambda g\}_{\lambda\in\Lambda}=\{ T_{(a_1-d_1) m+(a_2-d_2)n}M_{c_1 m+c_2n}g_{m,n}\}_{m,n\in \Z^2}\subseteq L^2(\R)\]
is a set of Gabor molecules \cite[Definition 10]{BCHL06b}) localized with respect to the lattice
\[
\tilde\Lambda=\begin{pmatrix} a_1-d_1&a_2-d_2\\
c_1&c_2
\end{pmatrix}\Z^{2}.
\]
For that purpose it suffices to show the existence of
$\Psi\in W(C,\ell^2)$, such that
$|\langle g_{m,n},T_xM_\omega\gamma\rangle|<\Psi(x,\omega)$ for all
$(m,n)\in\Z^2, (x,\omega)\in\R^{2}$. 
Here $\gamma(t)=e^{-t^2}$ and $W(C,\ell^2)$ denotes the Wiener amalgam space consisting of all continuous functions $f$ on $\R^2$ such that the norm 
\begin{equation*}
 \|f\|_{W(C,\ell^2)}=\left(\sum_{m\in\Z^2}\esssup_{z\in[0,1)^2}|f(z+m)|^2\right)^{\frac12}<\infty.
\end{equation*}
Note \eqref{eq:estimates1} implies
\begin{equation*}
\begin{aligned}
|\langle g_{m,n}, T_x M_{\omega}\gamma\rangle|&
\le\langle |g_{m,n}|,T_{x}|\gamma|\rangle=|g_{m,n}|\ast\gamma(x)\\
&\le \|g\|_{M^\infty(\R)}\cdot\phi_1\ast \gamma(x)\, , \\
|\langle \calF g_{m,n}, M_{-x} T_{\omega}\gamma\rangle|&\le\langle |g_{m,n}|,T_{\omega}|\gamma|\rangle=|g_{m,n}|\ast\gamma(\omega)\\ & \le
\|g\|_{M^\infty(\R)}\cdot \phi_2\ast\gamma(\omega),
\end{aligned}
\end{equation*}
see \cite{Gro04,Pfa08}.
Hence, by setting
\[h(t)=\|g\|_{M^\infty(\R)}\max\{\phi_1\ast \gamma(t),\phi_1\ast \gamma(-t),\phi_2\ast\gamma(t), \phi_2\ast\gamma(-t)\},\]
we obtain
\[|\langle g_{m,n},T_xM_\omega\gamma\rangle|\le h(\max\{|x|,|\omega|\})=h(\| (x,\omega)\|_\infty)\] 
with
$|h( t)|=O(t^{-s}),s>2$.
Hence, there exists a constant $c$ such that
\[|\langle g_{m,n},T_xM_\omega\gamma\rangle|\le c\cdot h(\| (x,\omega)\|)\] that can be bounded in turn pointwise by a function $\Psi(x,\omega)\in W(C,\ell^2)$. Thus, $\{H_\lambda g\}_{\lambda\in\Lambda}$ is a set of Gabor molecules.

Assume that $\spa(H_0,\Lambda)$ is identifiable by $g\in M^\infty(\R)$.
By employing \cite[Theorem 8]{BCHL06b} and \cite[Theorem 3]{BCHL06a} (the latter result is a Gabor molecule extension of Theorem~\ref{thm:identification1stExample}) we obtain that the Beurling density of $\tilde\Lambda$, given by
\begin{equation*}
 D(\tilde\Lambda)=\left|\det
\begin{pmatrix}
 a_1-d_1&a_2-d_2\\
c_1&c_2
\end{pmatrix}
\right|^{-1}
\end{equation*}must be less than 1 in order for $\{H_\lambda g\}_{\lambda\in\Lambda}$ to be Riesz in $L^2(\R)$.
A computation shows that this condition follows from $\bar D(\Lambda)>\sqrt{2}$, yielding the bound in Theorem \ref{thm:mainresult}.
\hfill $\square$

\begin{remark} \rm 
To obtain similarly a necessary density condition for $\{H_\lambda g\}_{\lambda\in\Lambda}$ to be a frame, we would have to
show that $|(a_1-d_1)c_2-(a_2-d_2)c_1|=D(\tilde\Lambda)^{-1}<1$ follows from $\bar D(\Lambda)>c$ for some positive constant $c$.
But this is not reasonable to expect, as increasing $b_1$ and/or $b_2$ greatly in 
\begin{equation*}
\begin{aligned}\bar D(\Lambda)=&\big[(a_1b_2-a_2b_1)^2+(a_1c_2-a_2c_1)^2\\ &\quad +(a_1d_2-a_2d_1)^2+(b_1c_2-b_2c_1)^2\\
&\quad +(b_1d_2-b_2d_1)^2+(c_1d_2-c_2d_1)^2\big]^{-1/2}.
\end{aligned}
\end{equation*}
would decrease $\bar D(\Lambda)$.
\end{remark}

\section{Examples of identifiable operator classes, design of identifiers}\label{section:sufficient}

To establish identifiability of $\spa(H_0,\Lambda)$ we seek  an identifier $g$ such
that any choice of coefficients $\{c_{\lambda}\}\in \ell^{2}(\Lambda)$ in $H=\sum_{\lambda\in\Lambda} c_{\lambda}H_\lambda$ can be computed from $H g$. 
Equivalently, we require that
$\{c_{\lambda}\}$ can be computed from the values of the Gabor coefficients  
$v_{\mu}=\langle H g,\pi_1(\mu)\gamma\rangle$, $\mu\in\mathcal M$,  where $(\gamma, \mathcal M)$ is an $L^2$-Gabor frame
for appropriately chosen window $\gamma\in L^2(\R)$ and lattice $\mathcal M\subset\R^2$. 
Then, our task is to solve the system of linear
equations
\begin{equation*}
  v_{\mu}=\langle H g,\pi_1(\mu)\gamma\rangle
    = \sum_{\lambda\in\Lambda}c_{\lambda}\,
     \langle H_\lambda g,\pi_1(\mu)\gamma\rangle
    =\sum_{\lambda\in\Lambda}A_{\mu;\lambda}\ c_{\lambda},\quad \mu\in\mathcal M
\end{equation*}
for $\{c_\lambda \}$. 
The doubly infinite matrix $A$ has entries $A_{\mu;\lambda}=\langle H_\lambda g,\pi_1(\mu)\gamma\rangle$.

If $g$ is such that the map 
$A\colon \ell^2(\Lambda)\rightarrow\ell^{2}(\mathcal M)$ has a bounded left inverse then $\spa(H_0,\Lambda) $ is identifiable. 

The design of identifiers $g$ can  be carried out on the coefficient level. 
In fact, when $(\tilde\gamma,\widetilde{ \mathcal M})$ is an appropriately chosen
Gabor frame for $L^2(\R)$, or, for example, an $\ell^\infty$-frame for $M^\infty(\R)$ \cite{AST01}, then we seek a coefficient sequence $\{d_{\widetilde \mu}\}$ so that the bi-infinite matrix with entries
\[A_{\mu;\lambda}
=\sum_{\widetilde \mu\in\widetilde {\mathcal M}}d_{\widetilde \mu} \langle H_{\lambda}\pi_1(\widetilde\mu)\widetilde\gamma
,\pi_1(\mu)\gamma\rangle\] is invertible.

To illustrate the method outlined above, we give an alternative proof of one implication in Theorem 3.1 in~\cite{KP06}. Also see \cite{Pfa11} for a comprehensive treatment of sampling and identification in operator Paley-Wiener spaces.

\begin{theorem} \label{th:example1}
The operator Paley-Wiener space $OPW^2([0,a]{\times}[0,\frac 1a])
$ 
is identifiable.
\end{theorem}
\begin{proof}
By definition the given operator space consists of
operators $H\in HS(\R)$ such that
\[\eta_H\in \spa\{M_{\frac ka, \ell a} \chi_{[0,a]{\times}[0,\frac 1a]},\ k,\ell \in\Z \},\]
that is, we consider
\[\Lambda=\begin{pmatrix} 0&0&0&\frac1a\\
0&0&a&0
\end{pmatrix}^T\Z^{2}.\]
Since $\{e^{2\pi i
(\frac{kx}a+a\ell y)}:(x,y)\in [0,a]{\times}[0,\frac 1a],k,\ell\in\Z\}$ forms
an orthonormal basis for the space $L^2([0,a]{\times}[0,\frac 1a])$ \cite{Kat76},
  the spreading function of any
operator $H\in OPW([0,a]{\times}[0,\frac 1a])$ has a unique expansion
\begin{equation*}
\eta_H=\sum_{k,\ell\in\Z}c_{k,\ell}M_{\frac k a,a\ell}\eta_0
\end{equation*}
with $\eta_0(x,\omega)=\chi_{[0,a]}(x)\chi_{[0,\frac 1a]}(\omega)$.

Set $c_{k,\ell,m,n}=c_{k,l}\delta_{0,0}(m,n)$.
We choose $\gamma=a^{-\frac12}\chi_{[0,a]}$ and observe that the Gabor system 
$(\gamma,a\Z\times \frac 1a\Z)$ is an orthonormal basis for $L^2(\R)$.
Using the  
formal relationship $\langle H g,f\rangle_{L^2(\R)}=\langle \eta_H,V_gf\rangle_{L^2(\R)}$  
\cite[Lemma 3.2]{GPR10} with $g=\delta_{a\Z}=\sum_{n\in\Z}\delta_{na}\in M^\infty(\R)$\footnote{Note that the inner product is still well-defined since $\eta_H$ has compact support.} we compute
\begin{equation*}
\begin{aligned}
A_{p,q;k,\ell}&=\langle H_{k,\ell}\delta_{a\Z},M_{\frac pa}T_{aq}\gamma\rangle \\
&=\langle T_{-aq,-\frac pa}M_{\frac ka -{\frac pa},a\ell}\eta_0, V_{\delta_{a\Z}}\gamma\rangle.\end{aligned}
\end{equation*}

The Zak transform $Z_a$ satisfies the relations
\[V_{\delta_{a\Z}}a^{-\frac12}\chi_{[0,a]}=Z_a a^{-\frac12}\chi_{[0,a]}\] and \[a^{-\frac12}Z_a\chi_{[0,a]}(x,\omega)=e^{2\pi i a\left[\frac xa\right]\omega},\]
for which we refer to \cite[Chapter 8]{Gro01}.
Then
\begin{equation*}
A_{p,q;k,\ell}
 =
\iint \chi_{[0,a]}(x+aq)e^{2\pi i\frac{(k-p)(x+aq)}a}\chi_{[0,\frac1a]}(\omega+\tfrac pa)e^{2\pi i\left(\omega+\frac pa\right)a\ell}e^{-2\pi ia\left[\frac xa\right]\omega}\,dx\,d\omega
\end{equation*}

We make the substitutions $y=x+aq,\, z=\omega+\tfrac{p}a$ and
note that since the integrand is nonzero for $aq\le x<aq+a$,
$\left[\tfrac xa\right]=q$, it follows that
\begin{equation*}
\begin{aligned}
A_{p,q;k,\ell} &=
 \int_0^a\int_0^{\frac 1a}e^{2\pi i\frac{(k-p)y}a}e^{2\pi i a\ell z-2\pi i aq( z-\frac pa)}\,dy\,d z\\
 &=
 e^{2\pi i (pq)}\int_0^a e^{2\pi i\frac{(k-p)y}a}dy\times \int_0^{\frac 1a} e^{2\pi ia(\ell-q) z}\,d z\\
 &= \delta_{p,q}(k,\ell),
\end{aligned}
\end{equation*}
where we used that   $\{e^{2\pi i\frac
nat}:n\in\Z\}$ and $\{e^{2\pi imat}:m\in\Z\}$ are orthonormal
bases for $L^2[0,a]$ and $L^2[0,\tfrac 1a]$, respectively.

The matrix $A=(A_{p,q;k,\ell})_{p,q;k,\ell\in\Z}$ is  the
identity, and thus invertible, which is what we had to prove.
%
%
\end{proof}

Recall that  
$V_g f(z)= \langle
f,\pi_1(z)g\rangle,\, z\in\R^2
$ is short-time Fourier transform of $f\in L^2(\R)$ with respect to the window $g\in L^2(\R)$ \cite[Chapter 3]{Gro01}.
An additional positive identifiability result is the following.

\begin{proposition}
Let $h,\{g_\lambda\}_{\lambda\in\Lambda}\in L^2(\R)$ be such that $\{g_\lambda\}$ is a Riesz sequence for its closed linear span in $L^2$, $\|h\|_{L^2(\R)}=1$. Then the operator family $\spa\{H_\lambda,\eta_{H_\lambda}=V_hg_\lambda\}$, is identifiable.
\end{proposition}
\begin{proof}
The Riesz sequence property of $\{g_\lambda\}$ implies that $\{\eta_{H_\lambda}\}$ is a Riesz sequence as well since it is the image of a Riesz sequence under the unitary map $V_h:g_\lambda\mapsto \eta_{H_\lambda}$ \cite{Chr03}. Consider the associated action of the operator $H_\lambda$ on $g$: $H_\lambda g=g_\lambda\langle g,h\rangle$. Therefore, the entries of the matrix $A$ have the form
\begin{equation}\label{eq26}
\begin{aligned}
A_{\mu;\lambda}&=\sum_{\tilde \mu\in\tilde\calM}d_{\tilde \mu}\langle H_\lambda\pi_1(\tilde\mu)\tilde\gamma, \gamma_\mu\rangle =\sum_{\tilde \mu\in\tilde\calM}d_{\tilde \mu}\langle g_\lambda, \gamma_\mu\rangle\langle\pi_1(\tilde\mu)\tilde\gamma,h\rangle.\end{aligned}
\end{equation}
Whenever $\{\gamma_\mu\}$ is chosen biorthogonal to $\{g_\lambda\}$, \eqref{eq26} becomes
\[A_{\mu;\lambda}=\delta(\mu-\lambda)\sum_{\tilde \mu\in\tilde\calM}d_{\tilde \mu}\langle\pi_1(\tilde\mu)\tilde\gamma,h\rangle,\]
which shows that the matrix $A$ is diagonal with non-zero entries for appropriate $\{d_{\tilde\mu}\}$. Hence, $A$ is invertible.\end{proof}

The following examples address more involved rank-2 lattices, for simplicity of calculation, we will consider Gaussian kernels only.

\begin{example}\label{prop:cased6gaussian}
Let $H_0$ be given by $\kappa_0(x,\omega)=e^{-\pi(x^2+\omega^2)}$, that is, $\eta_0(t,\nu)=\frac1{\sqrt2}e^{-\pi i\sqrt{2}
t\nu}e^{-\frac{\pi}2(t^2+\nu^2)}$.
\begin{enumerate} \item Let
$\Lambda=\left(\begin{smallmatrix}
\alpha &0&0&0\\
0 &\beta&\alpha&0\end{smallmatrix}\right)^T\Z^2$. If $\alpha,\beta$ are such
that $|\alpha(\beta+\alpha\sqrt2)|\ge\sqrt2,|\alpha\beta|>\sqrt2,|\alpha|>1$,
then the operator family $ \spa(H_0,\Lambda)$ is identifiable.

\item Let
$\Lambda=\left(\begin{smallmatrix}
{\alpha} &0&0& 0 \\
0&0&\alpha&\beta
\end{smallmatrix}\right)^T\Z^2$. If $\alpha,\beta$ are such that
$|\alpha|>1$ and $\frac{\beta\sqrt2}\alpha\in\Q$, then the operator
family $ \spa(H_0,\Lambda)$ is
identifiable.
\end{enumerate}
\end{example}

\begin{figure}
\begin{center}
\includegraphics[width=9cm]{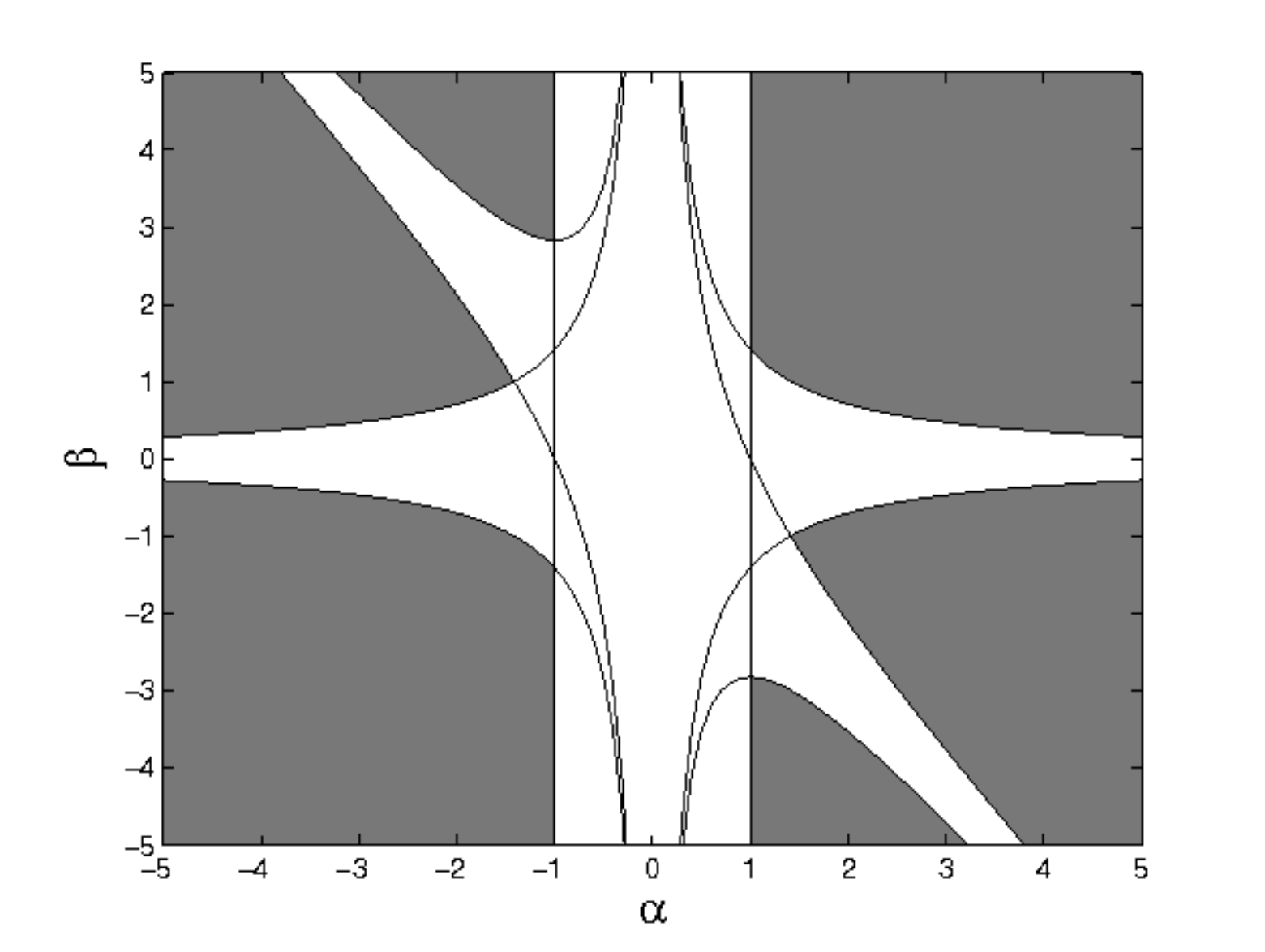}
\end{center}
\caption{The set $(\alpha,\beta)$ fulfilling the conditions in Example~\ref{prop:cased6gaussian}, Statement 1, is represented by the shaded region. The region encompasses the intersection of the convex hull of the parabolas $|\alpha(\beta+\sqrt2\alpha)|\ge\sqrt2$, $|\alpha\beta|>2$ with the subset of the plane $|\alpha|>1$.} \label{figure:hyperbola1}
\end{figure}

The pairs $(\alpha,\beta)$ satisfying the conditions in Example~\ref{prop:cased6gaussian}, Statement 1, are illustrated in Figure \ref{figure:hyperbola1}. Statement 2 of Example~\ref{prop:cased6gaussian} indicates that identifiability may depend on conditions that are not expressible in terms of density or simple inequalities. When $\frac{\beta\sqrt2}\alpha\notin\Q$, the Gabor system $(\eta_0,\Lambda)$ is not a Riesz sequence, so the problem is not considered in this paper.

The following example shows that $\bar D(\Lambda)$ being small does not necessarily guarantee identifiability of $ \spa(H_0,\Lambda)$.
\begin{example}\label{prop:casenotidef}
Let $H_0$ be given by $\eta_0\in M_s^1(\R^2)$, $ s>2$, and let
$\Lambda=\left(\begin{smallmatrix}
0&0&0&\beta\\
{\alpha} &\beta&0& 0
\end{smallmatrix}\right)^T\Z^2$. If $|\alpha\beta|<1$, then the operator
family $ \spa(H_0,\Lambda)$ is not
identifiable.
\end{example}

The condition $|\alpha\beta|<1$ cannot be expressed in terms of 2-Beurling density of the index set $\Lambda$, $\bar D(\Lambda_i)=\frac1{|\beta|\sqrt{\alpha^2+\beta^2}}$. 
In fact, for any $\epsilon>0$, we can find $\alpha,\beta$ with $|\alpha\beta|<1$ such that $\bar D(\Lambda)<\epsilon$. 
For instance, when $\alpha=10^{10},\beta=(10^{10}+1)^{-1}$, $|\alpha\beta|<1$, so
the family $ \spa(H_0,\Lambda)$ is not identifiable, but $\bar D(\Lambda)\approx10^{-20}$ is very small.

Further examples of identification using the approach described in this paper are given in \cite[Sections 4.4 - 4.6]{GPR10}.


%
%

%
%


%
%









\begin{thebibliography}{99}
\bibitem{AST01}A. Aldroubi, Q. Sun and W.S. Tang, $p$-frames and shift invariant subspaces of {$L^p$},
    {\it J. Four. Anal. Appl.}, {\bf 7}(1), 1-21, 2001.

\bibitem{BCHL06a}R. Balan, P.G. Casazza, C. Heil and Z. Landau, Density, overcompletness, and localization of frames. I: Theory,
    {\it J. Four. Anal. Appl.}, {\bf 12}(2), 105-143, 2006.

\bibitem{BCHL06b}R. Balan, P.G. Casazza, C. Heil and Z. Landau, Density, overcompletness, and localization of frames. II: Gabor systems,
    {\it J. Four. Anal. Appl.}, {\bf 12}(3), 307-344, 2006.

\bibitem{Bel69}P.A. Bello, Measurement of random time-variant linear channels.
    {\it IC}, {\bf 15}, 469-475, 1969.

\bibitem{Chr03}O. Christensen, {\it An Introduction to Frames and Riesz Bases}, Birkh\"auser, Boston, 2003.

\bibitem{Con90}J.B. Conway, {\it A Course in Functional Analysis}, Springer Verlag, New York, 1990.

\bibitem{Fei83}H.G. Feichtinger, Modulation spaces on locally compact abelian groups,
    {\it Tech. report, Univ. Vienna, Dept. Math.}, 1983.

\bibitem{Fei89}H.G. Feichtinger, Atomic characterizations of modulation spaces through Gabor-type representations,
    {\it Rocky Mount. J. Math.}, {\bf 19}, 113-126, 1989.

\bibitem{FG96}H.G. Feichtinger and K. Gr\"ochenig, Gabor frames and time-frequency analysis of distributions,
    {\it J. Funct. Anal.}, {\bf 146}(2), 464-495, 1996.

\bibitem{FN03}K. Gr\"ochenig and K. Nowak, A first survey of Gabor multipliers, in
    {\it Advances in Gabor analysis}, Birkh\"auser, Boston, 2003.

\bibitem{Fol99}G.B. Folland, {\it Real Analysis}, John Wiley \& Sons Inc., New York, 1999.

\bibitem{GP08}N. Grip and G.E. Pfander, A discrete model for the efficient analysis of time-varying narrowband communications channels,
    {\it Multidim. Syst. Signal Process.}, {\bf 19}(1), 3-40, 2008.

\bibitem{GPR10}N. Grip, G.E. Pfander and P. Rashkov, Identification of time-frequency localized operators,
    {\it Tech. report No. 20, Jacobs University Bremen}, 2010.

\bibitem{Gro01}K. Gr\"ochenig, {\it Foundations of Time-Frequency Analysis}, Birkh\"auser, Boston, 2001.

\bibitem{Gro04}K. Gr\"ochenig, Localization of frames, Banach frames, and the invertibility of the frame operator,
    {\it J. Four. Anal. Appl.}, {\bf 10}(2), 105-132, 2004.

\bibitem{GL03}K. Gr\"ochenig and M. Leinert, Wiener's lemma for twisted convolution and Gabor frames,
    {\it J. Am. Math. Soc.}, {\bf 17}(1), 1-18, 2003.
    
 \bibitem{Hei07}C. Heil, History and evolution of the density theorem for {G}abor frames, {\it J. Four. Anal. Appl.}, {\bf 12}, 113-166, 2007.

\bibitem{Kai59}T. Kailath, Sampling models for linear time-variant filters, MIT, Research Laboratory of Electronics, 1959.

\bibitem{Kat76}Y. Katznelson, {\it An Introduction to Harmonic Analysis}, Dover, New York, 1976.

\bibitem{KP06}W. Kozek and G.E. Pfander, Identification of operators with bandlimited symbols,
    {\it SIAM J. Math. Anal.}, {\bf 37}(3), 867-888, 2006.

\bibitem{Kut06}G. Kutyniok, Beurling density and shift-invariant weighted irregular Gabor systems,
    {\it STSIP}, {\bf 5}(2), 163-181, 2006.

\bibitem{Pfa08}G.E. Pfander, On the invertibility of ``rectangular'' bi-infinite matrices and applications in time-frequency analysis,
    {\it Lin. Alg. Appl.}, {\bf 429}, 331-345, 2008.

\bibitem{Pfa07c}G.E. Pfander, Sampling of operators,
    http://arxiv.org/abs/1010.6165, preprint, 2010.

\bibitem{Pfa11}G.E. Pfander, Sampling of operators,
    preprint, 2011.

\bibitem{PR10}G.E. Pfander and P. Rashkov, Window design for multivariate Gabor frames on lattices,
    {\it Tech. report No. 21, Jacobs University Bremen}, 2010.

\bibitem{PRW12}G.E. Pfander, P. Rashkov and Y. Wang, A geometric construction of tight multivariate Gabor frames with compactly supported smooth windows,
    {\it J. Four. Anal. Appl.}, online, 2012.
\bibitem{PW06}G.E. Pfander and D. Walnut, Operator identifcation and Feichtinger's algebra,
    {\it STSIP}, {\bf 5}(2), 183-200, 2006.
\end{thebibliography}
\end{document}